\documentclass[11pt]{amsart}
\usepackage{hyperref}
\usepackage{fullpage}
\usepackage{amsmath, amsthm, amssymb, graphicx, dsfont}
\usepackage[T1]{fontenc}
\usepackage{textcomp}
\usepackage{lmodern}
\usepackage[english]{babel}
\usepackage{microtype}
\usepackage{stmaryrd}
\expandafter\def\csname opt@stmaryrd.sty\endcsname
{only,shortleftarrow,shortrightarrow}
\usepackage{extpfeil}

\usepackage{amscd}
\usepackage{mathrsfs}
\usepackage{tikz-cd}
\usetikzlibrary{matrix}
\usepackage[colorinlistoftodos]{todonotes}
\usepackage{color}
\usepackage{bbm}
\usepackage{verbatim}
\usepackage[mathscr]{euscript}
\usepackage{hyperref}
\hypersetup{
	colorlinks=true,
	citecolor=blue,
	linkcolor=blue,
	filecolor=magenta,      
	urlcolor=blue,
}

\theoremstyle{plain}
\newtheorem{thm}{Theorem}[section]
\newtheorem*{thm*}{Theorem}

\newtheorem{prop}[thm]{Proposition}
    
\newtheorem{cor}[thm]{Corollary}       
\newtheorem{lem}[thm]{Lemma}

\theoremstyle{definition} 
\newtheorem{defin}[thm]{Definition}

\newtheorem{rem}[thm]{Remark}

\newcommand{\D}{\mathsf{D}}

\renewcommand{\L}{\mathcal{L}}
\newcommand{\CS}{\mathcal{S}}

\newcommand{\Ch}{\mathrm{Ch}}
\newcommand{\adicR}{\mathrm{Ch}(R,I)}
\newcommand{\adicS}{\mathrm{Ch}(S,J)}
\newcommand{\LR}{\mathrm{Ch}(R,L_0^I)}
\newcommand{\LS}{\mathrm{Ch}(S,L_0^J)}

\newcommand{\exta}{\theta_*}
\newcommand{\resa}{\theta^*}
\renewcommand{\mod}[1]{\mathrm{Mod}_{#1}}

\newcommand{\adicmod}{{\widehat{\mathrm{Mod}}}\hspace{0pt}^{\,I}_R}

\newcommand{\Hom}{\mathrm{Hom}}
\newcommand{\RHom}{\mathbb{R}\mathrm{Hom}}
\newcommand{\Lotimes}{\otimes^{\mathbb{L}}}
\newcommand{\Z}{\mathbb{Z}}
\renewcommand{\epsilon}{\varepsilon}
\begin{document}
\begin{abstract}
Given a commutative ring $R$ and finitely generated ideal $I$, one can consider the classes of $I$-adically complete, $L_0^I$-complete and derived $I$-complete complexes. Under a mild assumption on the ideal $I$ called weak pro-regularity, these three notions of completions interact well. We consider the classes of $I$-adically complete, $L_0^I$-complete and derived $I$-complete complexes and prove that they present the same homotopy theory. Given a ring homomorphism $R \to S$, we then give necessary and sufficient conditions for the categories of complete $R$-complexes and the categories of complete $S$-complexes to have equivalent homotopy theories. This recovers and generalizes a result of Sather-Wagstaff and Wicklein on extended local (co)homology.
\end{abstract}

\title{The homotopy theory of complete modules}
\author{Luca Pol}
\address[Pol]{Fakult\"{a}t f\"{u}r Mathematik, Universit\"{a}t Regensburg, Universit\"{a}tsstra{\ss}e 31, 93053 Regensburg, Germany}
\email{luca.pol@ur.de}
\author{Jordan Williamson}
\address[Williamson]{Department of Algebra, Faculty of Mathematics and Physics, Charles University in Prague, Sokolovsk\'{a} 83, 186 75 Praha, Czech Republic}
\email{williamson@karlin.mff.cuni.cz}
\maketitle
\setcounter{tocdepth}{1}

\section{Introduction}
Let $R$ be a commutative ring and $I$ be an ideal of $R$. The $I$-adic completion of an $R$-module $$M_I^\wedge = \underset{\leftarrow}{\lim} \,M/I^n M$$ is a fundamental tool in commutative algebra as it provides a formal way to consider infinitesimal neighbourhoods. This geometric interpretation motivates the study of the category of $I$-adically complete $R$-modules, that is, those modules $M$ for which the canonical map $M \to M^\wedge_I$ is an isomorphism. 
This category admits kernels and cokernels but is not abelian (nor even exact) in general and this makes these modules difficult to work with. In addition, the completion functor suffers 
from bad homological behaviours; for example, it is neither left nor right exact in general.

Therefore for homological purposes, it is often convenient to consider its left derived functors $L_n^I$ which were first studied by Matlis~\cite{Matlis78} and Greenlees-May~\cite{GreenleesMay92}. Of particular note is the zeroth left derived functor $L_0^I$ which is right exact. By replacing the $I$-adic completion functor with $L_0^I$, we are led to study the category of $L_0^I$-complete modules, that is, of those modules $M$ for which the canonical map 
$M \to L_0^I M$ is an isomorphism. This is an abelian category which contains all of the $I$-adically complete modules and by work of Salch~\cite{Salch10}, is in fact the universal abelian extension of the category of $I$-adically complete modules, see Proposition~\ref{prop:exact} for a precise statement. Working with this extension provides good homological properties and often significant advantages. 

The $I$-adic completion and $L_0^I$-completion functors admit evident extensions to chain complexes of $R$-modules by levelwise application. In particular, we can consider complexes of 
$I$-adically complete modules and $L_0^I$-complete modules, and see them as objects in the unbounded derived category $\D(R)$. In the derived category, a good homological replacement for $I$-adic completion is its total left derived functor $\mathbb{L}(-)_I^\wedge$; this is defined using dg-projective resolutions and therefore it can be difficult to compute with. This often makes the study of complexes $M$ for which the canonical map $M \to \mathbb{L}(M)_I^\wedge$ is a quasi-isomorphism inaccessible.

However under a mild condition on the ideal $I$ called weak pro-regularity, the total left derived functor of $I$-adic completion is naturally quasi-isomorphic to a more tangible construction $\Lambda_I$, which we call the derived completion~\cite[5.25]{MGM}. For a complex $M$, the derived completion is defined by $\Lambda_IM = \RHom_R(K_\infty(I),M)$ where $K_\infty(I)$ is the stable Koszul complex of Definition~\ref{defn:koszul}. In light of this, it is natural to consider the 
category of derived complete complexes, namely those complexes for which 
$M \to \Lambda_I M$ is a quasi-isomorphism. 

The goal of this paper is to relate these categories of complete complexes. There are inclusions
\[\text{$I$-adically complete complexes} \hookrightarrow \text{$L_0^I$-complete complexes}\hookrightarrow \text{derived $I$-complete complexes}.\]
Each of these categories is symmetric monoidal, but otherwise they exhibit significantly different behaviour. Nonetheless, in this paper we show that each of these categories present the same homotopy theory. The following theorem appears in the body of the paper as Theorem~\ref{thm:main1}. 
\begin{thm}
Let $R$ be a commutative ring and $I$ be a weakly pro-regular ideal. There are symmetric monoidal Quillen equivalences 
\[I\text{-adically complete complexes} \simeq_Q L_0^I\text{-complete complexes} \simeq_Q \text{derived $I$-complete complexes}.\]  
\end{thm}
More precisely, we show that that the category of $I$-adically complete complexes is symmetric monoidally Quillen equivalent to the category of derived $I$-complete complexes. 
The right hand Quillen equivalence was already known, see~\cite[3.7]{BHV20} for an $\infty$-categorical approach and \cite[6.10]{PolWilliamson} for a model categorical one. 

We also investigate how base change along a map of commutative rings interacts with the the Quillen equivalences of the previous theorem. The following theorem appears in the body of the paper as Theorem~\ref{thm:main2}.
\begin{thm}
Let $\theta\colon R \to S$ be a map of commutative rings and let $I$ and $J$ be weakly pro-regular ideals of $R$ and $S$ respectively such that $\sqrt{IS} = \sqrt{J}$. The vertical adjunctions in the diagram 
\[\begin{tikzcd}
\begin{tabular}{c} \text{$I$-adically complete} \\ \text{$R$-complexes} \end{tabular} \arrow[rr, hookrightarrow, yshift=-1mm] \arrow[ddd, "(-)_{J}^\wedge\circ\exta"', xshift=-1mm] \arrow[dr, yshift=-1mm] &  & \begin{tabular}{c} \text{$L_0^I$-complete} \\ \text{$R$-complexes} \end{tabular} \arrow[ddd, "L_0^{J}\exta"', xshift=-1mm] \arrow[ld, yshift=-1mm] \arrow[ll, "(-)_I^\wedge"', yshift=1mm]
 \\
 & \begin{tabular}{c} \text{derived $I$-complete} \\ \text{$R$-complexes} \end{tabular} \arrow[ul, "(-)_I^\wedge"', yshift=1mm] \arrow[d, "\exta"', xshift=-1mm] \arrow[ru, yshift=1mm, "L_0^I"] & \\
 & \begin{tabular}{c} \text{derived $J$-complete} \\ \text{$S$-complexes} \end{tabular}\arrow[dl, "(-)_{J}^\wedge"', yshift=1mm] \arrow[dr, yshift=1mm, "L_0^{J}"] \arrow[u, xshift=1mm, "\resa"'] & \\
\begin{tabular}{c} \text{$J$-adically complete} \\ \text{$S$-complexes} \end{tabular} \arrow[ur, yshift=-1mm] \arrow[rr, hookrightarrow, yshift=-1mm]\arrow[uuu, xshift=1mm, "\resa"'] &  & \begin{tabular}{c} \text{$L_0^J$-complete} \\ \text{$S$-complexes} \end{tabular} \arrow[uuu, xshift=1mm, "\resa"'] \arrow[lu, yshift=-1mm] \arrow[ll, "(-)_{J}^\wedge"', yshift=1mm]
\end{tikzcd} \]
are symmetric monoidal Quillen equivalences if and only if any of the following equivalent conditions hold:
\begin{itemize}
\item[(a)] $\Lambda_IR \to \Lambda_J S$ is an isomorphism in $\D(R)$;
\item[(b)] $K_\infty(I) \to K_\infty(J)$ is an isomorphism in $\D(R)$;
\item[(c)] $R_I^\wedge \to S_J^\wedge$ is an isomorphism of rings;
\item[(d)] $R/I \to S/IS$ is an isomorphism of rings.
\end{itemize}
\end{thm}

Sather-Wagstaff and Wicklein~\cite[4.13]{SWW16} show that if $R$ is Noetherian and $I$ is an ideal in $R$, then the category of derived $I$-complete $R$-complexes is equivalent to the category of derived $\widehat{I}$-complete $R_I^\wedge$-complexes where $\widehat{I}$ denotes the ideal $I\cdot R_I^\wedge$ of $R_I^\wedge$. Our theorem recovers and extends this result by considering change of base along an arbitrary ring homomorphism, and by providing a converse. This also gives the analogous result about derived torsion complexes which in particular recovers~\cite[4.12]{SWW16}, see Remark~\ref{rem:MGM}.  

\subsection*{Contribution of this paper and related work}
There has been a wealth of work on the various forms of completion and their interactions. Here we do not attempt to give a complete history, but to point out several key papers whose results are used here and which influenced our approach. Greenlees-May~\cite{GreenleesMay92, GreenleesMay95b} undertook research on the homological algebra of the left derived functors of $I$-adic completion and derived completion, with a particular focus on its relation with equivariant homotopy theory. Dwyer-Greenlees~\cite{DwyerGreenlees02} investigated derived torsion and completion (defined as certain Bousfield localizations) in a general setting using Morita theory. Alonso-Jeremias-Lipman~\cite{AJL} studied the derived completion in a geometric setting, and Schenzel~\cite{Schenzel03} and Porta-Shaul-Yekutieli~\cite{MGM} extended this work in the algebraic setting. Positselski~\cite{Positselski16, Positselski17, Positselski} has proved many results about the different forms of completions, and in particular, their relation to contramodules. We also note the work of Barthel-Heard-Valenzuela~\cite{BHV20} who have investigated derived completion and $L$-completion for comodules over a flat Hopf algebroid. Finally, we note the work of Lurie~\cite[\S 7]{SAG} which extends
some of the previously described work to the setting of ring spectra. 

The relationship between these papers in the literature is not entirely clear - for instance, the use of $L_0^I$-completion seems restricted to authors coming from a topological standing, and contramodules from the algebraic side. In fact, these two notions are the same, and this observation clarifies many relations between these papers, see Section~\ref{sec:contramodules} for more details. As such, in addition to the main results already described in the introduction, we hope that this paper collects together various helpful statements from the aforementioned papers and provides clear connections between the topological and algebraic sides. 

\subsection*{Conventions}
We always write the left adjoint on the top in an adjoint pair displayed horizontally, and on the left in an adjoint pair displayed vertically. We use $\mathbb{L}$ and $\mathbb{R}$ to denote total left and right derived functors respectively; for instance, $- \otimes_R^\mathbb{L} -$ denotes the derived tensor product of $R$-modules, and $\mathbb{R}\mathrm{Hom}_R(-,-)$ denotes the derived hom of $R$-modules.

Throughout the paper we work with commutative rings, ideals and complexes of modules. Instead, one could work with graded commutative rings, homogeneous ideals and dg-modules. For concreteness, we choose to work in the ungraded setting, but all the statements and proofs given easily transfer to the case when $R$ is a graded commutative ring. 

\subsection*{Acknowledgements}
We are grateful to Liran Shaul and Niall Taggart for their comments on a preliminary version of this paper, and to John Greenlees and Greg Stevenson for several helpful suggestions and discussions. The first author thanks the SFB 1085 Higher Invariants in Regensburg for support.
The second named author was supported by the grant GA~\v{C}R 20-02760Y from the Czech Science Foundation.

\section{\texorpdfstring{$I$}{I}-adic completion}
  In this section we recall a few important facts about $I$-adic completion. 
  For more details see~\cite[\S 10]{AM16}.  

\begin{defin}
Let $R$ be a commutative ring and $I$ be an ideal. 
The $I$-\emph{adic completion} of an $R$-module $M$ is given by $M_I^\wedge = \underset{\leftarrow}{\lim} \,M/I^n M.$ 
We say that the $R$-module $M$ is:
\begin{itemize}
\item \emph{$I$-adically complete} if the canonical map $\gamma_M\colon M \to M_I^\wedge$ is an isomorphism. 
\item \emph{$I$-adically separated} if the canonical map $\gamma_M\colon M \to M_I^\wedge$ is a monomorphism.
\end{itemize}
We write $\mod{R}$ for the category of $R$-modules and $\adicmod$ for the full subcategory of $I$-adically complete modules.
\end{defin}

\subsection{Topologies} We fix a commutative ring $R$ and an ideal $I$.
\begin{defin}
  The $I$-\emph{adic topology} on an $R$-module $M$ has a 
  basis of open sets given by the cosets $m + I^n M$ for all $m \in M$ and 
  $n\in\mathbb{N}$. 
  An $I$-\emph{adic module} is an $R$-module equipped with the $I$-adic topology.
  We denote by $\mod{R}^{\mathrm{adic}}$ the category of $I$-adic modules and 
  continuous module maps.
\end{defin}

\begin{lem}\label{lem:eq-adic}
  Let $f \colon M \to N$ be an $R$-module map. Then $f$ is continuous 
  with respect to the $I$-adic topology on $M$ and $N$. 
  In particular, the inclusion functor
  $ j\colon \mod{R}^{\mathrm{adic}} \to \mod{R}$ is an equivalence of 
  categories with inverse functor given by $(-)_{\mathrm{adic}}$ the 
  $I$-adic topology functor.
\end{lem}

\begin{proof}
  As we can always translate, it is enough to show that $f$ is 
  continuous at $0 \in M$. 
  Since $f$ is an $R$-module map, we have $I^n M \subset f^{-1}(I^n N)$. 
  This shows that there is an open neighbourhood of $0 \in M$ which 
  is sent into $I^n N$ by $f$. Therefore $f$ is continuous at $0 \in M$. 
  It follows that $\mod{R}^{\mathrm{adic}}$ sits inside $\mod{R}$ as 
  a full subcategory. As every $R$-module can be equipped with the $I$-adic topology,
  the inclusion functor is essentially surjective, 
  and hence an equivalence. 
\end{proof}

  The $I$-adic completion of an $R$-module $M$ also has an 
  \emph{inverse limit topology} which is defined as the limit topology of the 
  discrete spaces $M/I^n M$.  Using the universal property of the inverse limit 
  we obtain the following result.
  
\begin{lem}\label{lem-adic-universal}
  For any $I$-adically complete $R$-module $N$ and $R$-module 
  map $f \colon M \to N$, there exists a unique continuous $R$-module map 
  $\widehat{f} \colon M^{\wedge}_I \to N$ making the diagram 
  \[
  \begin{tikzcd}
   M \arrow[r, "f"] \arrow[d, "\gamma_M"'] & N \\
   M^{\wedge}_I \arrow[ru, "\widehat{f}"'] & 
  \end{tikzcd}
  \]
  commute. Here $M^{\wedge}_I$ and $N$ are endowed with the inverse limit 
  topologies.
\end{lem}

If the ideal is finitely generated, then the $I$-adic and inverse limit topology coincide.
  
\begin{prop}[{\cite[3.6]{Yekutieli11}}]\label{prop:algebraic and analytic}
  Let $I$ be a finitely generated ideal in a commutative ring $R$. 
  Then the $I$-adic topology and the inverse limit topology on $M^{\wedge}_I$ 
  coincide. Moreover, $M^{\wedge}_I$ is $I$-adically complete. 
\end{prop}

As a consequence we obtain the folllowing.
  
\begin{lem}\label{lem:adic adjunction}
  Let $R$ be a commutative ring and $I$ be a finitely generated ideal. 
  There exists an adjoint pair
  \[\begin{tikzcd}
  \mod{R} \arrow[rr, yshift=1mm, "(-)_I^\wedge"] \arrow[rr, hookleftarrow, yshift=-1mm, "i"'] & & \adicmod.
  \end{tikzcd} \]
\end{lem}

\begin{proof}
  By Proposition~\ref{prop:algebraic and analytic} the inverse limit topology 
  coincides with the $I$-adic topology. The existence of the adjoint pair then 
  follows from Lemmas~\ref{lem:eq-adic} and~\ref{lem-adic-universal}.
\end{proof}

\subsection{Additional structure}
  Let $I$ be a finitely generated ideal. 
  The category $\adicmod$ has all limits and colimits. 
  Using Lemma~\ref{lem:adic adjunction} we see that limits are calculated in the 
  underlying category of $R$-modules, and colimits are calculated by completing 
  the colimits in the underlying category of $R$-modules. In particular, 
  $\adicmod$ has kernels and cokernels where the latter need to be 
  completed. However, it is not an abelian category in general as images might not 
  coincide with coimages, see~\cite[4.3]{Vanderpool12} for an example.   
  One might still hope that $\adicmod$ is exact. 
  However, Simon~\cite[2.5]{Simon09} 
  shows that $\adicmod$ is not closed under extensions, and 
  therefore does not form an exact category where the exact sequences are the 
  usual short exact sequences of $R$-modules.
  
The tensor product of $I$-adically complete $R$-modules need not be $I$-adically complete. However, we will now show that $\adicmod$ admits the structure of a closed symmetric monoidal category, with the tensor product given by the completed tensor product.
  
\begin{lem}\label{lem:internal hom}
Let $R$ be a commutative ring, $I$ be a finitely generated ideal, and $M$ and $N$ be $R$-modules.
\begin{itemize}
\item[(a)] If $N$ is $I$-adically complete then $\mathrm{Hom}_R(M,N)$ is $I$-adically complete.
\item[(b)] The natural map $(M \otimes_R N)_I^\wedge \to (M_I^\wedge \otimes_R N_I^\wedge)_I^\wedge$ is an isomorphism.
\end{itemize}
\end{lem}
\begin{proof}
Taking a free presentation of $M$ gives an exact sequence $$0 \to \mathrm{Hom}_R(M,N) \to \prod_JN \to \prod_KN$$ from which part (a) follows, since $I$-adically complete modules are closed under kernels and products.

Let us now prove part (b). By symmetry of the tensor product, it is enough to check that $(M \otimes_R N)_I^\wedge \to (M_I^\wedge \otimes_R N)_I^\wedge$ is an isomorphism. By the Yoneda lemma, it suffices to check that the induced map $$\mathrm{Hom}_R((M_I^\wedge \otimes_R N)_I^\wedge, Z) \to \mathrm{Hom}_R((M \otimes_R N)_I^\wedge, Z)$$ is an isomorphism for all $I$-adically complete $Z$. By adjunction, this is equivalent to checking that $$\mathrm{Hom}_R(M_I^\wedge, \mathrm{Hom}_R(N, Z)) \to \mathrm{Hom}_R(M, \mathrm{Hom}_R(N,Z))$$ is an isomorphism. Since $\mathrm{Hom}_R(N,Z)$ is $I$-adically complete by part (a) the result follows.
\end{proof}

\begin{prop}\label{prop:adic-sym-mon}
Let $R$ be a commutative ring and $I$ be a finitely generated ideal. The category $\adicmod$ is closed symmetric monoidal, with tensor product given by $(M \otimes_R N)_I^\wedge$ and internal hom $\mathrm{Hom}_R(M,N).$ In particular, the $I$-adic completion functor $(-)_I^\wedge\colon \mod{R} \to \adicmod$ is strong symmetric monoidal.
\end{prop}
\begin{proof}
By Lemma~\ref{lem:internal hom}(b) the completed tensor product gives $\adicmod$ a symmetric monoidal structure. The coherence diagrams can be verified to commute using the coherence diagrams for $\mod{R}$. By Lemma~\ref{lem:internal hom}(a), the internal hom of $I$-adically complete modules is $I$-adically complete. Therefore the completed tensor product is left adjoint to the internal hom $\Hom_R(-,-)$ in $\adicmod$ by Lemma~\ref{lem:adic adjunction}.
\end{proof}

\subsection{Homological properties}
Finally we record some homological properties of the $I$-adic completion functor, which we will use later.
\begin{prop}\label{prop:homological-prop-adic}
Consider an ideal $I$ in a commutative ring $R$.
\begin{itemize}
 \item[(a)] If $M \to N$ is a surjective map of $R$-modules, then 
 $M^{\wedge}_I \to N^\wedge_I$ is also surjective.
 \item[(b)] The $I$-adic completion functor is neither left nor right exact.
 \item[(c)] If $R$ is Noetherian, then the $I$-adic completion functor is exact on 
 finitely generated $R$-modules.
 \item[(d)] If $0 \to M \to N \to F \to 0$ is a short exact sequence
 and $F$ is flat, then $0 \to M_I^\wedge \to N_I^\wedge \to F_I^\wedge \to 0$ is a 
 short exact sequence.
\end{itemize}
\end{prop}

\begin{proof}
 Parts (a) and (d) can be found in~\cite[Tag 0315]{stacks-project}. Part (c) 
 can be found in~\cite[10.12]{AM16}. Finally for a counterexample for (b) 
 see~\cite[Tag 05JF]{stacks-project}.
\end{proof}

\section{\texorpdfstring{$L_0^I$}{L}-completion}\label{sec:L completion}
The $I$-adic completion functor is neither left nor right exact in general. In order to do homological algebra, it is convenient to replace the $I$-adic completion by its zeroth left derived functor.

\begin{defin}
The $L_0^I$-completion functor $L_0^I$ is the zeroth left derived functor of $I$-adic completion. For an $R$-module $M$, we have the explicit formula $L_0^IM = \mathrm{coker}((P_1)^{\wedge}_I \to( P_0)^{\wedge}_I)$ where $P_*$ is a projective resolution of $M$.  More generally, we will write $L_n^I $ for the $n$-th left derived functor of $I$-adic completion.
\end{defin}

\begin{rem}\leavevmode
\begin{itemize}
\item[(a)] It is  standard that $L_n^I$ is independent of the chosen projective 
resolution.
\item[(b)] By construction $L_0^I$ is right exact.
\item[(c)] Since the $I$-adic completion functor is in general neither left nor right exact, the $L_0^I$-completion functor is not isomorphic to the $I$-adic completion functor.
\end{itemize}
\end{rem}

\begin{defin}
An $R$-module $M$ is said to be \emph{$L_0^I$-complete} if the natural map $\eta\colon M \to L_0^IM$ is an isomorphism. We write $\mod{R}^{L_0^I}$ for the full subcategory of $\mod{R}$ consisting of the $L_0^I$-complete $R$-modules. 
\end{defin}

\subsection{Key properties}
We will often need to impose a condition on the ideal $I$ called \emph{weak pro-regularity}. We note that weakly pro-regular ideals are finitely generated by definition, but since we shall not require the precise definition, we refer the reader to~\cite{Schenzel03,MGM} for details. This assumption ensures that there is a good interplay between the different notions of completion, and is often satisfied in practice. For instance, any ideal in a Noetherian ring is weakly pro-regular, see~\cite[4.34]{MGM} and~\cite[\S 1]{GreenleesMay92}. 

\begin{lem}[{\cite[4.1]{GreenleesMay92}}]\label{lem:higher L vanishes}
	Let $R$ be a commutative ring and $I$ be a weakly pro-regular ideal. If $M$ is $N_I^\wedge$ or $L_0^IN$ for some $R$-module $N$, then $M$ is $L_0^I$-complete and $L_i^IM = 0$ for all $i \geq 1$.
\end{lem}

\begin{lem}[{\cite[A.2(a)]{HoveyStrickland99}}]\label{lem-factorization}
Let $R$ be a commutative ring, $I$ be a weakly pro-regular ideal and $M$ be an $R$-module. The canonical map $\gamma \colon M \to M^{\wedge}_I$ factors as $M \xrightarrow{\eta}L_0^I M \xrightarrow{\varepsilon} M^{\wedge}_I$. Furthermore, 
$\varepsilon$ is always surjective. 
\end{lem}

\begin{rem}
The references given for the previous two lemmas work under stricter hypotheses than weak pro-regularity which are only imposed to ensure that the left derived functors of $I$-adic completion can be calculated by the homology of the derived completion functor. This holds under the weaker assumption of weak pro-regularity by work of Porta-Shaul-Yekutieli~\cite{MGM}, see Theorem~\ref{thm:local homology and L-completion} and Remark~\ref{rem-meaning-L_n-complexes} for a precise statement.
\end{rem}

Using these we can measure the difference between $I$-adic completion and $L_0^I$-completion.
\begin{prop}\label{prop:adic is separated and L}
	Let $R$ be a commutative ring and $I$ be a weakly pro-regular ideal. An $R$-module $M$ is $I$-adically complete if and only if it is $L_0^I$-complete and $I$-adically separated.
\end{prop}
\begin{proof}
	The backward implication follows from the fact that the natural map 
	$\varepsilon \colon L_0^IM \to M_I^\wedge$ is an epimorphism by Lemma~\ref{lem-factorization}. The forward implication follows from Lemma~\ref{lem:higher L vanishes}.
\end{proof}

\begin{cor}\label{cor:factor}
Let $R$ be a commutative ring and $I$ be a weakly pro-regular ideal. The triangle of adjunctions 
\[\begin{tikzcd}
\adicmod \arrow[rr, hookrightarrow, yshift=-1mm] \arrow[dr, yshift=-1mm, hookrightarrow] &  & \mod{R}^{L_0^I} \arrow[ll, "(-)_I^\wedge"', yshift=1mm]
 \\
 & \mod{R} \arrow[ul, "(-)_I^\wedge"', yshift=1mm] \arrow[ru, yshift=1mm, "L_0^I"] \arrow[ur, yshift=-1mm, hookleftarrow] &
\end{tikzcd}\]
commutes. That is, the underlying diagram of left (or, equivalently right) adjoints commutes up to natural isomorphism.
\end{cor}
\begin{proof}
From Proposition~\ref{prop:adic is separated and L}, one sees that the category of $I$-adically complete modules embeds into the category of $L_0^I$-complete modules. Using Lemma~\ref{lem:adic adjunction}, one verifies that this inclusion has a left adjoint given by the $I$-adic completion. The triangle of inclusions then clearly commutes which gives the result.
\end{proof}
	
\subsection{Additional structure}
We now record some key facts about the category of $L_0^I$-complete modules. For more details, we direct the reader to~\cite[\S 5]{PolWilliamson} and the appendix of~\cite{HoveyStrickland99}.
\begin{prop}\label{prop:L properties} Let $R$ be a commutative ring and $I$ be a weakly pro-regular ideal.
\begin{itemize}
\item[(a)] The category of $L_0^I$-complete modules is an abelian category with kernels and cokernels calculated in the underlying category of $R$-modules. 
\item[(b)] The category of $L_0^I$-complete modules is a closed symmetric monoidal category with tensor product $L_0^I(M \otimes_R N)$ and internal hom $\mathrm{Hom}_R(M,N).$
\item[(c)] There is an adjunction 
$$\begin{tikzcd}
\mod{R} \arrow[r, "L_0^I", yshift=1mm] \arrow[r, hookleftarrow, yshift=-1mm] & \mod{R}^{L_0^I} 
\end{tikzcd} $$
and the left adjoint is strong symmetric monoidal.
\item[(d)] The category of $L_0^I$-complete modules has all limits and colimits. The limits are calculated in the underlying category of $R$-modules and the colimits are given by the $L_0^I$-completion of the colimit in the underlying category of $R$-modules.
\end{itemize}
\end{prop}

Recall that a full subcategory of $\mod{R}$ is said to be \emph{exact} if it 
is abelian and the inclusion functor is exact, and is said to be \emph{replete} if it is closed under isomorphisms. The category of $L_0^I$-complete 
modules satisfies the following universal property.

\begin{prop}\label{prop:exact}
Let $R$ be a commutative ring and $I$ be a weakly pro-regular ideal. The category of $L_0^I$-complete $R$-modules is the smallest replete 
  exact full subcategory of $\mod{R}$ which contains all $I$-adically complete 
  $R$-modules. 
\end{prop}
\begin{proof}
The claim is proved for $R$ Noetherian by Salch~\cite[4.13]{Salch10}. The more general statement follows from Salch's general machinery~\cite[3.14]{Salch10} since $(-)_I^\wedge$ preserves surjections by Proposition~\ref{prop:homological-prop-adic} and $L_1^IL_0^IM = 0$ for all $R$-modules $M$ by Lemma~\ref{lem:higher L vanishes}.
\end{proof}

\subsection{Extension to complexes}\label{sec:exttocomplexes}
We can readily extend the $I$-adic completion functor and its left derived functors $L_n^I$ to complexes by applying the functor levelwise. As such, we make the following definition.
\begin{defin}
Let $M$ be a complex of $R$-modules.
\begin{itemize}
\item We say that $M$ is \emph{$I$-adically complete} if the natural map $M \to M_I^\wedge$ is an isomorphism, that is, if $M$ is a complex of $I$-adically complete $R$-modules.  
\item We say that $M$ is \emph{$L_0^I$-complete} if the natural map $M \to L_0^IM$ is an isomorphism, that is, if $M$ is a complex of $L_0^I$-complete $R$-modules.  
\end{itemize}
\end{defin}

\begin{cor}\label{cor:Lcompletehomology}
Let $R$ be a commutative ring and $I$ be a weakly pro-regular ideal. If $M$ is an $L_0^I$-complete complex, then each homology group $H_nM$ is $L_0^I$-complete.
\end{cor}
\begin{proof}
By Proposition~\ref{prop:exact} the inclusion functor $\mod{R}^{L_0^I} \hookrightarrow \mod{R}$ is exact.
\end{proof}

\section{The total left derived functor of \texorpdfstring{$I$}{I}-adic completion}\label{sec:totalleft}
Throughout we let $\D(R)$ denote the (unbounded) derived category of the commutative ring $R$. 

\begin{rem}\label{rem-dual-perfect-comp}
 We recall that a complex $M \in \D(R)$ is perfect (i.e., quasi-isomorphic to a bounded complex of finitely generated projective modules) if and only if it is compact in 
 the sense that $\RHom_R(M, -)$ commutes with arbitrary coproducts. This also 
 implies that the canonical map
 \[
 \RHom_R(M, R) \Lotimes_R N \to \RHom_R(M,N)
 \]
 is a quasi-isomorphism, for all $N \in \D(R)$.
\end{rem}

\begin{defin}\label{def-dg-projective-flat}
\leavevmode 
\begin{itemize}
\item A complex $P$ is \emph{dg-projective} if $\Hom_R(P,-)$ preserves surjective quasi-isomorphisms. 
\item A \emph{dg-projective resolution} for a complex $M$ is a choice of quasi-isomorphism $P \to M$ where $P$ is a dg-projective complex.
\end{itemize}
\end{defin}

We record the following key fact from~\cite[9.6.1]{AvramovFoxbyHalperin03}.
\begin{lem}\label{lem:dgprojimpliesproj}
Let $M$ be a complex of $R$-modules. If $M$ is dg-projective then each term $M_n$ of the complex is projective, but the converse fails.
\end{lem} 

We have seen in Proposition~\ref{prop:homological-prop-adic} that the $I$-adic 
completion functor is neither left nor right exact. For homological purposes 
it is convenient to consider its total left derived functor. 

\begin{defin}
The total left derived functor of the $I$-adic completion $\mathbb{L}(-)_I^\wedge\colon \D(R) \to \D(R)$ is defined by $\mathbb{L}(M)_I^\wedge = P_I^\wedge$ where $P \xrightarrow{\sim} M$ is a dg-projective resolution of $M$. This comes equipped with a natural map $\gamma^\mathbb{L}\colon M \to\mathbb{L}(M)_I^\wedge$.
\end{defin}

\begin{rem}\label{rem-meaning-L_n-complexes}\leavevmode
\begin{itemize}
  \item[(a)] If $M$ is a complex, $H_n(\mathbb{L}(M)_I^\wedge)$ should not be confused with $L_n^I M$ as defined in Section~\ref{sec:exttocomplexes}. However, note that if $M$ is a module then the two notions do coincide.
  \item[(b)] We note that the complexes $M$ for which the map $\gamma^\mathbb{L}\colon M \to\mathbb{L}(M)_I^\wedge$ is a quasi-isomorphism are called \emph{$I$-adically cohomologically complete} in~\cite{MGM}.
  \end{itemize}
\end{rem}

 Any complex in $\D(R)$ admits a dg-projective resolution~\cite[9.3.3]{AvramovFoxbyHalperin03} so the 
 total left derived functor $\mathbb{L}(-)^\wedge_I$ always exists. 
 Furthermore, it is independent of the chosen dg-projective resolution. For more details, see~\cite[\S 3]{MGM} and~\cite{AJL}.

\begin{rem}
It would be enough to calculate $\mathbb{L}(-)_I^\wedge$ using the weaker notions of a $K$-projective resolution~\cite[1.1]{Spaltenstein}, or of a dg-flat or $K$-flat resolution by~\cite[3.6]{MGM}.  Since the dg-projective objects are the cofibrant objects in the projective model structure on $\Ch(R)$ (see Proposition~\ref{prop:proj-model}), we choose throughout to work with dg-projective resolutions instead. 
\end{rem}

\section{Koszul complexes}\label{sec:koszul}
In this section we collect a few results about Koszul complexes required for the rest of the paper. 

\begin{defin}\label{defn:koszul}
Let $R$ be a commutative ring and $I = (x_1,\ldots,x_n)$ a finitely generated ideal. 
\begin{itemize}
\item The \emph{unstable Koszul complex} denoted $K(I)$ is $$K(I) = K(x_1) \otimes_R \cdots \otimes_R K(x_n)$$ where $K(x_i)$ is the complex $R \xrightarrow{x_i} R$ in degrees 0 and -1.
\item The \emph{stable Koszul complex} denoted $K_\infty(I)$ is defined by $$K_\infty(I) = K_\infty(x_1) \otimes_R \cdots \otimes_R K_\infty(x_n)$$ where $K_\infty(x_i)$ is the complex $R \to R[1/x_i]$ in degrees 0 and -1.
\end{itemize}
\end{defin}

\begin{rem}
We note that in our terminology, the stable Koszul complex is what some authors would call the \v{C}ech complex. This is an unfortunate (but well-established) difference in convention.
\end{rem}

\begin{rem}\label{rem:radical}
A priori, it appears that the stable Koszul complex depends on the chosen set of generators for $I$. In fact, the stable Koszul complex depends only on the radical of the ideal up to a zig-zag of natural quasi-isomorphisms by~\cite[1.2]{GreenleesMay95b} and the comments after it.
\end{rem}

We now introduce some useful notation and terminology. 

\begin{defin}
Consider complexes $M,N \in \D(R)$. 
\begin{itemize}
\item We denote by $\mathrm{Loc}(M)$ the localizing subcategory generated by $M$, i.e., the smallest replete full subcategory of $\D(R)$ containing $M$ which is closed under retracts, suspensions, completing triangles and arbitrary coproducts.
\item We say that $N$ is \emph{built} from $M$ (or $M$ \emph{builds} $N$) if $N \in \mathrm{Loc}(M)$.
\end{itemize}
\end{defin} 

We will use the following results frequently throughout the paper.
\begin{lem}[{\cite[5.4, 5.5]{Greenlees01b},~\cite[6.1]{DwyerGreenlees02}}]\label{lem-koszul-build}
Let $R$ be a commutative ring and $I$ be a finitely generated ideal. The unstable Koszul complex $K(I)$, the stable Koszul complex $K_\infty(I)$ and $R/I$ all build each other.
\end{lem}

\begin{prop}\label{prop-equivalent}
Let $R$ be a commutative ring and $I$ be a finitely generated ideal. The following are equivalent for a complex $M \in \D(R)$:
\begin{itemize}
\item[(a)] $K(I)\Lotimes_R M \simeq 0$;
\item[(b)] $K_\infty(I) \Lotimes_R M\simeq 0$;
\item[(c)] $R/I \Lotimes_R M \simeq 0$;
\item[(d)] $\RHom_R(K(I),M)\simeq 0$;
\item[(e)] $\RHom_R(K_\infty(I), M)\simeq 0$;
\item[(f)] $\RHom_R(R/I, M) \simeq 0$.
\end{itemize}
\end{prop}
\begin{proof}
 Consider the subcategory $\L$ of complexes 
 $N\in \D(R)$ such that $N \Lotimes_R M$ is acyclic. One easily checks  
 that $\L$ is a localizing subcategory. Since $K(I)$, $K_\infty(I)$ and $R/I$ 
 build each other by Lemma~\ref{lem-koszul-build}, we see that (a), (b) and (c) are 
 all equivalent. 
 Now consider the subcategory $\CS$ of complexes $N \in \D(R)$ such that 
 $\RHom_R(N,M)$ is acyclic. Again one checks that $\CS$ is localizing and so by 
 Lemma~\ref{lem-koszul-build} we see that (d), (e) and (f) are equivalent. 
 It is only left to show 
 that (a) is equivalent to (d). We note that the dual $DK(I) = \RHom_R(K(I),R)$ is 
 equivalent to $\Sigma^{n} K(I)$ where $n$ is the number of generators of $I$. 
 Moreover, as $K(I)$ is perfect 
 there is a canonical equivalence 
 \[
 \RHom_R(K(I), M) \simeq \Sigma^n K(I)\Lotimes_R M
 \]
 for all $M \in \D(R)$, see Remark~\ref{rem-dual-perfect-comp}. 
 This shows that (a) is equivalent to (d) which concludes 
 the proof.
\end{proof}

\section{Derived completion}\label{sec:derived completion}
In this section we introduce the derived completion functor and recall how this relates to the $I$-adic completion and $L_0^I$-completion functors. 

\begin{defin}
For a complex $M$, the derived completion functor $\Lambda_I\colon \D(R) \to \D(R)$ is defined by $\Lambda_IM = \mathbb{R}\mathrm{Hom}_R(K_\infty(I),M).$ A complex $M\in \D(R)$ is said to be \emph{derived complete} if the natural map $\lambda_M\colon M \to \Lambda_IM$ is a quasi-isomorphism. 
\end{defin}

\begin{rem}
We warn the reader that in~\cite{MGM} and other algebraic sources, $\Lambda_I$ denotes the $I$-adic completion functor which we instead denote by $(-)_I^\wedge$. Throughout the paper we have chosen to follow topological conventions where $\Lambda_I$ traditionally denotes the derived completion functor.
\end{rem}

The following result clarifies the relation between the derived completion and 
the total left derived functor of the adic completion. Recall from Section~\ref{sec:totalleft} that the total left derived functor $\mathbb{L}(-)_I^\wedge$ of the $I$-adic completion is calculated using a dg-projective resolution. 
\begin{thm}[{\cite[5.25]{MGM}}]\label{thm:local homology and L-completion}
Let $R$ be a commutative ring and $I$ be a weakly pro-regular ideal. For all $M\in\D(R)$, there is a natural quasi-isomorphism
	\[
	\tau_M \colon \Lambda_I (M) \xrightarrow{\sim} \mathbb{L}(M)_I^{\wedge}
	\]
	making the diagram
	\[
	\begin{tikzcd}
	& M \arrow[dl, "\lambda_M"'] \arrow[dr, "\gamma^\mathbb{L}_M"] & \\
	\Lambda_IM \arrow[rr, "\tau_M"'] & & \mathbb{L}(M)^{\wedge}_I
	\end{tikzcd}
	\]
	commute.
\end{thm}

The next result explains the relation between derived completion and $L_0^I$-completion.

\begin{prop}\label{prop:implication-completions}
Let $R$ be a commutative ring and $I$ be a weakly pro-regular ideal. For any complex $M$, there is a convergent spectral sequence 
\[
E^2_{p,q}= L^I_p(H_qM) \Rightarrow H_{p+q}(\Lambda_IM).
\]
Furthermore, the following implications hold:
\begin{itemize}
\item[(a)] A complex $M$ is derived complete if and only if each homology group $H_nM$ is $L_0^I$-complete. In particular, a module is derived complete if 
and only if it is $L_0^I$-complete.
\item[(b)] If each homology group $H_nM$ is $I$-adically complete, then $M$ is derived complete, but the converse fails in general. 
\item[(c)] If a complex $M$ is $L_0^I$-complete, then it is also derived complete. 
In particular, an $I$-adically complete complex is derived complete. 
\end{itemize}
\end{prop}

\begin{proof}
 For the existence of the spectral sequence see~\cite[3.3]{GreenleesMay95b}.
Part (a) follows from the spectral sequence; see~\cite[3.7]{BHV20},~\cite[6.15]{DwyerGreenlees02} 
 or~\cite[6.11]{PolWilliamson}. The first claim in part (b) follows from (a) and 
 the fact that an $I$-adically complete module is also $L_0^I$-complete by 
 Proposition~\ref{prop:adic is separated and L}. An explicit example of a derived 
 complete complex whose homology groups are not $I$-adically complete is given 
 in~\cite[4.33]{MGM}. For (c) note that the homology groups of $M$ are 
 $L_0^I$-complete by Corollary~\ref{cor:Lcompletehomology}. 
 It then follows from (a) that $M$ is derived complete. 
 The final part of (c) follows from 
 Proposition~\ref{prop:adic is separated and L}.
\end{proof}

\section{Contramodules and \texorpdfstring{$L_0^I$}{L}-completion}\label{sec:contramodules}
In this section, we show that $I$-contramodules and $L_0^I$-complete modules are in fact different points of view on the same objects. This has previously appeared in work of Positselski~\cite[2.2]{Positselski16}, but since the author does not use the language of $L_0^I$-complete modules we reproduce this result and give an alternative proof here. This observation clarified many things for the authors. In particular, Positselski has proved many results about the relation of derived completion with contramodules, see~\cite{Positselski16, Positselski17, Positselski}, but the relation with the work of Greenlees-May~\cite{GreenleesMay92} and Barthel-Heard-Valenzuela~\cite{Barthel, BHV20} was unclear to the present authors.
\begin{defin} Let $R$ be a commutative ring.
\begin{itemize}
\item Let $s \in R$. An $R$-module $M$ is said to be an \emph{$s$-contramodule} if $$\mathrm{Hom}_R(R[1/s],M) = 0 = \mathrm{Ext}_R^1(R[1/s],M).$$ where $R[1/s]$ is localization at the multiplicative set $\{1,s,s^2,\ldots\}$.
\item Let $I = (x_1,\ldots ,x_n)$ be a finitely generated ideal of $R$. An $R$-module is said to be an \emph{$I$-contramodule} if it is an $x_i$-contramodule for each $1 \leq i \leq n$.
\end{itemize}
\end{defin}

It follows from~\cite[5.1]{Positselski17} that the property of being 
an $I$-contramodule only depends on the ideal $I$ and not on the chosen generators. Using that $R[1/s]=R[1/s^n]$ for all $s\in I$ and $n>0$, we deduce that this property depends only on the radical of $I$.

\begin{prop}
Let $R$ be a commutative ring and $I$ be a weakly pro-regular ideal. An $R$-module $M$ is $L_0^I$-complete if and only if it is an $I$-contramodule.
\end{prop}
\begin{proof}
Let $I = (x_1,\ldots ,x_n)$. By~\cite[1.9(i)]{GreenleesMay92} we have $$L_0^IM = L_0^{(x_n)}L_0^{(x_1,\ldots, x_{n-1})}M = L_0^{(x_n)}L_0^{(x_{n-1})}\cdots L_0^{(x_1)}M,$$ from which one sees that it is sufficient to prove the claim in the principal case. This special case was proved by Yekutieli~\cite[Lemma 7]{Yekutieli15} and here we recall the argument. 

Let us assume that $I=(x)$ is principal and let $M$ be a $R$-module. By definition there is a 
cofibre sequence $K_{\infty}(I) \to R \to R[1/x]$ which induces an exact sequence 
\[
0 \to \Hom_R(R[1/x],M) \to M \to H_0(\mathbb{R}\Hom_R(K_{\infty}(I),M)) \to \mathrm{Ext}^1_R(R[1/x],M) \to 0.
\]
So $M$ is a contramodule if and only if $M \to H_0(\mathbb{R}\Hom_R(K_{\infty}(I),M))$ is an 
isomorphism. Therefore by Theorem~\ref{thm:local homology and L-completion}, $M$ is an $I$-contramodule if and only if $M$ is $L_0^I$-complete. 
\end{proof}

\section{Model structures}
In this section we equip the categories of chain complexes of complete modules with model structures. We refer the reader to~\cite{Hovey99} for the general 
background on model categories. 

For a commutative ring $R$ and an ideal $I$, we fix the following notation:
\begin{itemize}
\item $\Ch(R)$ denotes the category of (unbounded) chain complexes of $R$-modules;
\item $\adicR$ denotes the category of complexes of $I$-adically complete 
$R$-modules;
\item $\LR$ denotes the category of complexes of $L_0^I$-complete $R$-modules.
\end{itemize}
If $I$ is a finitely generated ideal, then we have inclusions
\[
\adicR \subset \LR \subset \Ch(R)
\]
by Corollary~\ref{cor:factor}. 

So that we can refer to it, let us recall the projective model structure on $\Ch(R)$.
\begin{prop}[{\cite[2.3.9, 2.3.11, 4.2.13]{Hovey99}}]\label{prop:proj-model}
There is a model structure on $\Ch(R)$ in which the weak equivalences are the quasi-isomorphisms, the fibrations are the surjections, and the cofibrations are the degreewise split monomorphisms which have dg-projective cokernel. Moreover, this is a symmetric monoidal model category.
\end{prop}

\subsection{Adic completion}
In this section we equip the category $\adicR$ of complexes of $I$-adically complete $R$-modules with a model structure. The procedure to this is similar to that used by the authors for $L_0^I$-complete modules~\cite{PolWilliamson}, see also Rezk~\cite{Rezk}. We remark that the argument we use here, heavily relies upon the good behaviour of the category of $L_0^I$-complete modules.
\begin{lem}\label{lem:adic completion preserves lifting properties}
Let $R$ be a commutative ring and $I$ be a finitely generated ideal.
\begin{itemize}
\item[(a)] The adic completion functor $(-)_I^\wedge$ takes cofibrations in $\Ch(R)$ to morphisms which have the left lifting property with respect to surjective quasi-isomorphisms of complexes of $I$-adically complete modules.
\item[(b)] The adic completion functor $(-)_I^\wedge$ takes acyclic cofibrations in $\Ch(R)$ to morphisms which have the left lifting property with respect to surjections of complexes of $I$-adically complete modules.
\end{itemize}
\end{lem}
\begin{proof}
We prove (a). The proof of (b) is similar. Suppose that $f\colon M \to N$ is a cofibration in $\Ch(R)$ and let $g\colon X \to Y$ be a surjective quasi-isomorphism in $\adicR$. We must show that there exists a lift as shown in the diagram
\[\begin{tikzcd}
M_I^\wedge \arrow[r] \arrow[d, "f_I^\wedge"'] & X \arrow[d, "g"] \\
N_I^\wedge \arrow[r] \arrow[ur, dashed] & Y.
\end{tikzcd}
\]
Since maps $N_I^\wedge \to X$ are in bijection with $R$-module maps $N \to X$ by Lemma~\ref{lem:adic adjunction}, it suffices to give a lift $N \to X$ in $\Ch(R)$. Such a lift exists since $f$ is a cofibration in $\Ch(R).$
\end{proof}

\begin{lem}\label{lem:factoring preserves completion}
Let $R$ be a commutative ring and $I$ be a weakly pro-regular ideal. Suppose that $M \to N$ is a cofibration in $\Ch(R)$. If $M$ is $I$-adically complete and $H_*N$ is $L_0^I$-complete, then the natural map $N \to N_I^\wedge$ is a quasi-isomorphism.
\end{lem}
\begin{proof}
Since $M \to N$ is a cofibration, the cokernel $P$ is dg-projective. Therefore by Lemma~\ref{lem:higher L vanishes}, the long exact sequence of derived functors collapses to a short exact sequence $$0 \to L_0^IM \to L_0^IN \to L_0^IP \to 0.$$ 
Note that $M$ is $L_0^I$-complete by Proposition~\ref{prop:adic is separated and L} and so by Corollary~\ref{cor:Lcompletehomology} it has $L_0^I$-complete homology. From the long exact sequence in homology we deduce that each $H_nP$ is $L_0^I$-complete too, and hence that $P$ is derived complete by Proposition~\ref{prop:implication-completions}(a). By Theorem~\ref{thm:local homology and L-completion}, the natural map $P \to P^{\wedge}_I$ is therefore a quasi-isomorphism. Since $P$ is a complex of flat $R$-modules by Lemma~\ref{lem:dgprojimpliesproj} and $M$ is $I$-adically complete, it follows from Proposition~\ref{prop:homological-prop-adic}(d) that $N \to N_I^\wedge$ is a quasi-isomorphism as claimed. 
\end{proof}

\begin{prop}\label{prop:adic model structure}
Let $R$ be a commutative ring and $I$ be a weakly pro-regular ideal. There is a model structure on $\adicR$ in which the weak equivalences are the quasi-isomorphisms, the fibrations are the surjections, and the cofibrations are the maps which have the left lifting property with respect to the acyclic fibrations. Moreover, this is a symmetric monoidal model category and the adjunction $$
\begin{tikzcd}
\adicR \arrow[r, hookrightarrow, yshift=-1mm] & \Ch(R) \arrow[l, "(-)_I^\wedge"', yshift=1mm]
\end{tikzcd}
$$
is a strong monoidal Quillen adjunction.
\end{prop}
\begin{proof}
The only axioms which require elaboration are the factorization axiom and the lifting axiom. For the factorization axiom we must show that any map in $\adicR$ can be factored in two ways:
\begin{enumerate}
\item into a map with the left lifting property with respect to acyclic fibrations followed by an acyclic fibration;
\item into a map with the left lifting property with respect to fibrations followed by a fibration.
\end{enumerate}
We prove the existence of factorization (1). The proof of the second is similar. 

Given a map $X \to Y$ in $\adicR$, one can factor $X \to Y$ into a cofibration $\alpha\colon X \to Z$ followed by an acyclic fibration $\beta\colon Z \to Y$ in $\Ch(R)$. By Lemma~\ref{lem:adic adjunction}, the map $\beta$ corresponds to a map $\nu\colon Z_I^\wedge \to Y$  which is still surjective and hence a fibration. It remains to show that $\nu$ is a quasi-isomorphism. Since $Z \to Y$ is a quasi-isomorphism and $Y$ is $I$-adically complete, $Z$ has $L_0^I$-complete homology by Corollary~\ref{cor:Lcompletehomology}. Therefore by Lemma~\ref{lem:factoring preserves completion}, $Z \to Z_I^\wedge$ is a quasi-isomorphism, and $\nu$ is a quasi-isomorphism by the 2-out-of-3 property. One sees that the factorization \[X \xrightarrow{\gamma\alpha} Z_I^\wedge \xrightarrow{\nu} Y\] is of the required form by Lemma~\ref{lem:adic completion preserves lifting properties}.

It remains to prove the lifting axiom. It is immediate from the definition that the cofibrations have the left lifting property with respect to the acyclic fibrations, so it remains to show that the acyclic cofibrations have the left lifting property with respect to the fibrations. This follows from the retract argument as follows. Let $f\colon A \to B$ be an acyclic cofibration and $g\colon M \to N$ be a fibration. By the factorizations proved above and the 2-out-of-3 property of weak equivalences, we can factor $f$ into a map $i\colon A \to C$ with the left lifting property with respect to fibrations followed by an acyclic fibration $p\colon C \to B$. By the lifting property of $f$ against $p$, we obtain a map $h\colon B \to C$. Since $i$ has the left lifting property with respect to fibrations, we obtain a map $h'\colon C \to M$. The composite $h'\circ h$ gives the required lift, and this completes the proof of the existence of the model structure.

Now that we have proved the existence of the model structure, it is clear that the adjunction in the proposition is Quillen, since the right adjoint preserves both the weak equivalences and fibrations by definition.

We now verify that this is a symmetric monoidal model category. Let $f\colon M \to N$ and $g\colon X \to Y$ be fibrations in $\adicR$. We must show that the pullback product map $$\mathrm{Hom}_R(N,X) \to \mathrm{Hom}_R(M,X) \times_{\mathrm{Hom}_R(M,Y)} \mathrm{Hom}_R(N,Y)$$ is a fibration which is acyclic if either $f$ or $g$ is. This holds in $\Ch(R)$ by Proposition~\ref{prop:proj-model} and so in $\adicR$ too, as the weak equivalences and fibrations are the same. 
Also note that $R^\wedge_I$ is cofibrant in $\adicR$ since $R$ is so in $\Ch(R)$ and the $I$-adic completion functor is left Quillen. The unit axiom of 
$\adicR$ immediately follows from this observation. It is only left to note that $(-)_I^\wedge\colon \Ch(R) \to \adicR$ is strong monoidal by Proposition~\ref{prop:adic-sym-mon}, so the adjunction is a strong monoidal Quillen adjunction as claimed.
\end{proof}

\subsection{\texorpdfstring{$L_0^I$}{L}-completion}
The category of complexes of $L_0^I$-complete modules has a projective model structure analogous to that of $I$-adically complete modules described in Proposition~\ref{prop:adic model structure}.
\begin{prop}[{\cite[6.4,~6.6]{PolWilliamson}}]
Let $R$ be a commutative ring and $I$ be a weakly pro-regular ideal. There is a symmetric monoidal model structure on $\LR$ in which the weak equivalences are the quasi-isomorphisms and the fibrations are the surjections.
\end{prop}

\subsection{Derived completion}
In this section we follow Greenlees-May~\cite{GreenleesMay95b} in putting a model structure on derived complete modules. This is given by a Bousfield localization of the category $\Ch(R)$.

\begin{defin}
Let $K(I)$ be the unstable Koszul complex. We say that a map $f\colon M \to N$ of complexes is an \emph{$K(I)$-equivalence} if the map $K(I) \otimes_R^\mathbb{L} f\colon K(I) \otimes_R^\mathbb{L} M \to K(I) \otimes_R^\mathbb{L} N$ is a quasi-isomorphism. The homological localization $L_{K(I)}\Ch(R)$ is the model structure on $\Ch(R)$ in which the weak equivalences are the $K(I)$-equivalences and the cofibrations are the cofibrations in $\Ch(R)$ (i.e., the monomorphisms which are linearly split and have dg-projective cokernel). This model structure exists by a similar argument to~\cite[VIII.1.1]{EKMM} and it is moreover a symmetric monoidal model category, see~\cite[3.11]{PolWilliamson} for instance. 
\end{defin}

The following result relates this model category to the category of derived complete complexes.

\begin{lem}\label{lem:derived completion imples K equivalence}
Let $R$ be a commutative ring and $I$ be a finitely generated ideal. A map $f\colon M \to N$ is a $K(I)$-equivalence if and only if $\Lambda_I(f)\colon \Lambda_I M \to \Lambda_I N$ is a quasi-isomorphism. Furthermore, the homotopy category of $L_{K(I)}\Ch(R)$ may be identified with the full subcategory of $\D(R)$ consisting of the derived complete complexes.
\end{lem}

\begin{proof}
Write $C$ for the cofibre of $M \to N$. By Proposition~\ref{prop-equivalent}, $\RHom_R(K_\infty(I), C) \simeq 0$ if and only if $K(I) \Lotimes_R C \simeq 0$ from which the first claim follows. In~\cite[4.2]{GreenleesMay95b} it is shown that $\Lambda_I$ is a fibrant replacement in $L_{K(I)}\Ch(R)$ which implies the second claim.
\end{proof}

We end this section by relating this model structure to the model structure on $I$-adically complete $R$-modules. 
\begin{prop}\label{prop:QA}
Let $R$ be a commutative ring and $I$ be a weakly pro-regular ideal. The adjunction
$$
\begin{tikzcd}
\adicR \arrow[r, hookrightarrow, yshift=-1mm] & L_{K(I)}\Ch(R) \arrow[l, "(-)_I^\wedge"', yshift=1mm]
\end{tikzcd}
$$
is a strong monoidal Quillen adjunction.
\end{prop}
\begin{proof}
Firstly, we proved in Proposition~\ref{prop:adic model structure} that \[
 \begin{tikzcd}
 \adicR \arrow[r, shift right,  "i"'] &  
 \Ch(R) \arrow[l,shift right, "(-)_I^\wedge"'] \end{tikzcd}\] is a strong monoidal Quillen adjunction. Since the cofibrations in the homological localization are the same as the cofibrations in the original model structure, $(-)_I^\wedge$ preserves cofibrations as a functor $L_{K(I)}\Ch(R) \to \adicR.$ 
 
It remains to show that this functor preserves acyclic cofibrations. Let $f\colon X \to Y$ be a cofibration in $\Ch(R)$ which is also a $K(I)$-equivalence. 
In particular this means that it has dg-projective cokernel $P$ and that $K(I) \otimes_R^\mathbb{L} P$ is acyclic. By Proposition~\ref{prop-equivalent}, we deduce that 
$\Lambda_IP \simeq 0$, and hence that $P_I^\wedge \simeq 0$ by Theorem~\ref{thm:local homology and L-completion}. As $P$ is dg-projective and hence termwise flat by Lemma~\ref{lem:dgprojimpliesproj}, Proposition~\ref{prop:homological-prop-adic}(d) shows that $X_I^\wedge \to Y_I^\wedge \to P_I^\wedge$ is short exact and therefore $f_I^\wedge\colon X_I^\wedge \to Y_I^\wedge$ is a quasi-isomorphism. This concludes the proof of the fact that the adjunction is a strong monoidal Quillen adjunction. 
\end{proof}

\section{The Quillen equivalences}\label{sec:QEs}
The goal of this section is to prove that the homotopy theories of $I$-adically complete, $L_0^I$-complete and derived complete complexes are all symmetric monoidally equivalent.

\begin{thm}\label{thm:cts QE}
Let $R$ be a commutative ring and $I$ be a weakly pro-regular ideal. There is a strong monoidal Quillen equivalence 
\[
 \begin{tikzcd}
 \adicR \arrow[r, shift right,  "i"'] &  
 L_{K(I)}\Ch(R). \arrow[l,shift right, "(-)_I^\wedge"'] \end{tikzcd}\]
\end{thm}
\begin{proof}
We proved that this is a strong monoidal Quillen adjunction in Proposition~\ref{prop:QA}. To show that it is a Quillen equivalence, it suffices to check that the derived unit and counit are weak equivalences in the respective model categories.

Let $M$ be a complex and choose a dg-projective resolution $P$ of $M$. Then the derived unit on $M$ is given by $P \to P_I^\wedge$ and we must check that this is a $K(I)$-equivalence. By Lemma~\ref{lem:derived completion imples K equivalence}, this is equivalent to checking that $\Lambda_I(P) \to \Lambda_I(P_I^\wedge)$ is an equivalence. Consider the commutative diagram 
\[\begin{tikzcd}
\Lambda_IP \arrow[dd, "\Lambda_I(\gamma_P)"'] \arrow[rr, "\tau_P", "\sim"'] & &  P_I^\wedge \arrow[dd, "\gamma_{P_I^\wedge}^\mathbb{L}"'] \arrow[dr, "\gamma_{P^\wedge_I}"] & \\
& & & \Lambda_I(P_I^\wedge) \arrow[dl, "\tau_{P^\wedge_I}", "\sim"']  \\
\Lambda_I(P_I^\wedge) \arrow[rr, "\tau_{P_I^\wedge}"', "\sim"] & & \mathbb{L}(P_I^\wedge)_I^\wedge 
\end{tikzcd} \] 
obtained from Theorem~\ref{thm:local homology and L-completion}. The same theorem 
tells us that the all the maps denoted by $\tau$ are quasi-isomorphisms. 
By Proposition~\ref{prop:implication-completions}(c) we know that $\gamma_{P^\wedge_I}$ is a quasi-isomorphism. The commutativity of the diagram and the 
$2$-out-of-$3$ property of quasi-isomorphisms gives the required claim.

Let $N$ be in $\adicR$ and choose a dg-projective resolution $c \colon Q\xrightarrow{\sim}N$ in $\Ch(R)$. Then the derived counit on $N$ is given by $Q_I^\wedge \to N$, and we need to check that this is a quasi-isomorphism. As $N$ is $I$-adically complete,  it is sufficient to check that $c_I^\wedge\colon Q_I^\wedge \to N_I^\wedge$ is a quasi-isomorphism. Consider the following two commutative 
diagrams
\[
\begin{tikzcd}
Q \arrow[r, "\lambda_Q"] \arrow[d, "\sim", "c"'] & \Lambda_IQ \arrow[d, "\Lambda_I(c)", "\sim"'] &   &  Q \arrow[r, "\sim"', "\gamma_Q"] \arrow[d, "\sim", "c"'] & Q_I^\wedge \arrow[d, "c_I^\wedge"]  \\ 
N \arrow[r, "\sim", "\lambda_N"'] & \Lambda_IN & & N \arrow[r, "\sim", "\gamma_N"'] &  N_I^\wedge.
\end{tikzcd}
 \]
By Proposition~\ref{prop:implication-completions}(d), we know that $N$ is derived complete so by the commutativity of the diagram on the left we see that $Q$ is derived complete. By Theorem~\ref{thm:local homology and L-completion} we know that $Q$ is $I$-adically complete, and so the result follows from the commutativity of the diagram on the right.
\end{proof}

\begin{thm}\label{thm:main1}
Let $R$ be a commutative ring and $I$ be a weakly pro-regular ideal. Each of the following adjunctions
\[\begin{tikzcd}
\adicR \arrow[rr, hookrightarrow, yshift=-1mm] \arrow[dr, yshift=-1mm, hookrightarrow] &  & \LR \arrow[ll, "(-)_I^\wedge"', yshift=1mm]
 \\
 & L_{K(I)}\Ch(R) \arrow[ul, "(-)_I^\wedge"', yshift=1mm] \arrow[ru, yshift=1mm, "L_0^I"] \arrow[ur, hookleftarrow, yshift=-1mm] &
\end{tikzcd}\]
is a symmetric monoidal Quilen equivalence.
\end{thm}
\begin{proof}
These were proved to be adjunctions in Corollary~\ref{cor:factor}. The Quillen equivalence $\LR \simeq_Q L_{K(I)}\Ch(R)$ is~\cite[6.10]{PolWilliamson} and the Quillen equivalence $\adicR \simeq_Q L_{K(I)}\Ch(R)$ is Theorem~\ref{thm:cts QE}. Therefore the remaining Quillen equivalence follows from the two-of-three property of Quillen equivalences.
\end{proof}

\section{Change of base}
Let $\theta\colon R \to S$ be a map of commutative rings and let $I$ and $J$ be ideals of 
$R$ and $S$ respectively. The ring homomorphism $\theta\colon R \to S$ induces an extension-restriction of scalars adjunction $(\theta_*,\theta^*)$. In this section, we investigate when this adjunction interacts well with $I$-completions of $R$-modules and $J$-completions of $S$-modules. More precisely, we give necessary and sufficient conditions for the vertical adjunctions shown in Figure~\ref{fig:diagram} to be symmetric monoidal Quillen equivalences.

\begin{figure}[!htb]
\centering 
\begin{tikzcd}
\adicR \arrow[rr, hookrightarrow, yshift=-1mm] \arrow[ddd, "(-)_{J}^\wedge\circ\exta"', xshift=-1mm] \arrow[dr, yshift=-1mm] &  & \LR \arrow[ddd, "L_0^{J}\exta"', xshift=-1mm] \arrow[ld, yshift=-1mm] \arrow[ll, "(-)_I^\wedge"', yshift=1mm]
 \\
 & L_{K(I)}\Ch(R) \arrow[ul, "(-)_I^\wedge"', yshift=1mm] \arrow[d, "\exta"', xshift=-1mm] \arrow[ru, yshift=1mm, "L_0^I"] & \\
 & L_{K(J)}\Ch(S)\arrow[dl, "(-)_{J}^\wedge"', yshift=1mm] \arrow[dr, yshift=1mm, "L_0^{J}"] \arrow[u, xshift=1mm, "\resa"'] & \\
\adicS \arrow[ur, yshift=-1mm] \arrow[rr, hookrightarrow, yshift=-1mm]\arrow[uuu, xshift=1mm, "\resa"'] &  & \LS \arrow[uuu, xshift=1mm, "\resa"'] \arrow[lu, yshift=-1mm] \arrow[ll, "(-)_{J}^\wedge"', yshift=1mm]
\end{tikzcd} 
\caption{\label{fig:diagram} Completions and change of base.}
\end{figure}

\subsection{Change of base and completion}

Firstly, we discuss how completion behaves with respect to restriction of scalars. 

\begin{lem}\label{lem-condition ideals}
Let $\theta\colon R \to S$ be a map of commutative rings and let $I$ and $J$ be finitely generated ideals of $R$ and $S$ respectively such that $\sqrt{IS} = \sqrt{J}$.
\begin{itemize}
\item[(a)] There is a natural isomorphism of functors $(-)^\wedge_{IS}\cong (-)^\wedge_J$.
\item[(b)] There is a natural isomorphism $K_\infty(J) \cong S \Lotimes_R K_\infty(I)$ in $\D(S)$.
\end{itemize}
\end{lem}

\begin{proof}
We first claim that there exist  $p,q\in \Z$ such that $IS^p \subset J^q \subset IS.$ Assuming this, we get systems of maps $M/J^{qk}M\to M/IS^{k}M$ and 
$M/ IS^{pk} M\to M/J^{qk}M$ for all $k\geq 0$ and modules $M$. Then one checks that the induced maps between completions are mutually inverse so part (a) follows. 

To prove the claim, we write $J=(y_1, \ldots, y_m)$. For each $y_j\in J$, there exists $a_j\in \Z$ such that $y_j^{a_j}\in IS$ since $J \subset \sqrt{IS}$. Put $a=\max_j a_j$ so that $y_j^a\in IS$ for all $j$.  Note that any generator of $J^q$ can be written as $y_1^{\alpha_1} \ldots y_m^{\alpha_m}$ with $\alpha_1 + \ldots+ \alpha_m = q$. If we take $q$ large enough (for instance $q=ma$), then there must exist at least one index $j$ such that $\alpha_j \geq a$ so $y_j^{\alpha_j}\in IS$. This shows that $J^q \subset IS$. Now notice that 
$\sqrt{J^q}=\sqrt{J}$, so by applying the same argument to $J^q$ and $IS$ we 
conclude that there exists $p\in \Z$ such that $IS^p \subset J^q$. This proves 
the claim and hence concludes the proof of (a). 

For part (b), one checks directly from the definitions that there is a natural isomorphism of 
complexes $K_\infty(IS) \cong S \otimes_R K_\infty(I)$. Since $K_\infty(I)$ is a bounded complex of flat $R$-modules it is $K$-flat~\cite[11.3.3]{AvramovFoxbyHalperin03}, and therefore $S \Lotimes_R K_\infty(I) = S \otimes_R K_\infty(I)$. The claim then follows from the fact the stable Koszul complex depends only on the radical of the ideal, see Remark~\ref{rem:radical}.
\end{proof}

\begin{lem}\label{lem:change of base}
Let $\theta\colon R \to S$ be a map of commutative rings and let $I$ and $J$ be finitely generated ideals of $R$ and $S$ respectively such that $\sqrt{IS} = \sqrt{J}$. Let $M$ be a complex of $S$-modules.
\begin{itemize}
\item[(a)] There is a natural isomorphism $(\resa M)_I^\wedge \cong \resa M_J^\wedge$ in $\adicR$.
\item[(b)] There is a natural isomorphism $\resa L_0^{J}(M) \cong L_0^I(\resa M)$
in $\LR$.
\item[(c)] There is a natural isomorphism $\resa\Lambda_JM \simeq \Lambda_I(\resa M)$ in $\D(R)$. 
\end{itemize}
\end{lem}
\begin{proof}
For part (a) we have $\resa M_J^\wedge \cong \resa M_{IS}^\wedge \cong (\resa M)_I^\wedge$ using Lemma~\ref{lem-condition ideals}(a) and the fact that restriction of scalars commutes with limits. Then part (b) follows from (a) and the fact that $\theta^*$ preserves cokernels. 
Part (c) follows from Lemma~\ref{lem-condition ideals}(b) and an application of the tensor-hom adjunction. 
\end{proof}

In light of the previous lemma, we will often neglect to write the restriction of scalars functor in what follows.

\begin{lem}\label{lem:radical}
Let $\theta\colon R \to S$ be a map of commutative rings and let $I$ and $J$ be finitely generated ideals of $R$ and $S$ respectively such that $\sqrt{IS} = \sqrt{J}$. There is an equality of model categories $$L_{K(J)}\Ch(S) =L_{K(IS)}\Ch(S).$$
\end{lem}
\begin{proof}
Since they both have the same cofibrations by definition, it suffices to show that a map is a $K(J)$-equivalence if and only if it is a $K(IS)$-equivalence. Since the stable Koszul complex depends on the ideal only up to radical, the $K_\infty(J)$-equivalences are the same as the $K_\infty(IS)$-equivalences. Using Proposition~\ref{prop-equivalent}, we see that the $K(J)$-equivalences are the same as the $K_\infty(J)$-equivalences. Similarly, one sees that the $K(IS)$-equivalences are the same as the $K_\infty(IS)$-equivalences and the result follows.
\end{proof}

\begin{prop}\label{prop:equivconditions}
Let $\theta\colon R \to S$ be a map of commutative rings and let $I$ and $J$ be weakly pro-regular ideals of $R$ and $S$ respectively such that $\sqrt{IS} = \sqrt{J}$. The following conditions are equivalent:
\begin{itemize}
\item[(a)] $\Lambda_IR \to \Lambda_J S$ is an isomorphism in $\D(R)$;
\item[(b)] $K_\infty(I) \to K_\infty(J)$ is an isomorphism in $\D(R)$;
\item[(c)] $R_I^\wedge \to S_J^\wedge$ is an isomorphism of rings;
\item[(d)] $R/I \to S/IS$ is an isomorphism of rings.
\end{itemize}
\end{prop}
\begin{proof}
Firstly let us show that (a) and (b) are equivalent. Write $C$ for the cofibre of $R \to S$ in $\D(R)$. Using Lemma~\ref{lem:change of base}(c), we see that (a) is equivalent to $\Lambda_I(C) \simeq 0$. By Proposition~\ref{prop-equivalent}, this is equivalent to $K_\infty(I) \Lotimes_R C \simeq 0$ and hence the map $K_\infty(I) \to K_\infty(I) \Lotimes_R S \simeq K_\infty(J)$ being an isomorphism in $\D(R)$.
The equivalence of (a) and (c) is an immediate consequence of the fact that $\Lambda_IR \simeq R_I^\wedge$ and $\Lambda_JS \simeq S_J^\wedge$ by Theorem~\ref{thm:local homology and L-completion}. Next we show that (d) implies (c). Since the isomorphism $R/I \to S/IS$ induces an isomorphism on each level of the limit system, there is an isomorphism $R_I^\wedge \to S_{IS}^\wedge$. Combining this with Lemma~\ref{lem-condition ideals}(a) shows that $R_I^\wedge \to S_J^\wedge$ is an isomorphism. 

It remains to show that (b) implies (d). We have already seen that (b) is equivalent to the statement that $K_\infty(I) \Lotimes_R C \simeq 0$, and hence to $R/I \Lotimes_R C \simeq 0$ by Proposition~\ref{prop-equivalent}. The latter condition is equivalent to $R/I \to R/I \Lotimes_R S$ being a quasi-isomorphism and therefore the homology of $R/I \Lotimes_R S$ is concentrated in degree zero. This shows that the natural map $R/I \Lotimes_R S \to R/I \otimes_R S \cong S/IS$ is a quasi-isomorphism, and hence the map $R/I \to S/IS$ is a quasi-isomorphism, and as such an isomorphism. 
\end{proof}

\begin{rem}
Write $\Gamma_I = K_\infty(I) \Lotimes_R -$ for the derived $I$-torsion 
functor, i.e., the (derived) local cohomology with respect to $I$. Then part (b) could also be written as $\Gamma_IR \to \Gamma_JS$ being an isomorphism in $\D(R)$.
\end{rem}

\begin{thm}\label{thm:main2}
Let $\theta\colon R \to S$ be a map of commutative rings and let $I$ and $J$ be weakly pro-regular ideals of $R$ and $S$ respectively such that $\sqrt{IS} = \sqrt{J}$. The vertical adjunctions in the diagram shown in Figure~\ref{fig:diagram}
are symmetric monoidal Quillen equivalences if and only if any of the equivalent conditions in Proposition~\ref{prop:equivconditions} hold. \end{thm}
\begin{proof}
First suppose that the vertical adjunctions are Quillen equivalences so that $$\begin{tikzcd}L_{K(I)}\Ch(R) \arrow[r, yshift=1mm, "\exta"] & L_{K(J)}\Ch(S) \arrow[l, yshift=-1mm, "\resa"]\end{tikzcd}$$ is a symmetric monoidal Quillen equivalence. In particular, the derived unit on $R$ is a weak equivalence in $L_{K(I)}\Ch(R)$ which means that $\theta\colon R \to S$ is a $K(I)$-equivalence. It follows by Lemma~\ref{lem:derived completion imples K equivalence} that $\Lambda_IR \to \Lambda_IS$ is a quasi-isomorphism, and hence that $\Lambda_IR \to \Lambda_JS$ is an isomorphism in $\D(R)$ by Lemma~\ref{lem:change of base}.

We now prove the converse. By Lemma~\ref{lem:radical} it is enough to show 
that 
$$\begin{tikzcd}L_{K(I)}\Ch(R) \arrow[r, yshift=1mm, "\exta"] & L_{K(IS)}\Ch(S) \arrow[l, yshift=-1mm, "\resa"]\end{tikzcd}$$ is a symmetric monoidal Quillen equivalence. We will prove this using the Left Localization Principle~\cite[3.13]{PolWilliamson}. Firstly, we need to check that $\theta^*$ 
sends $K(IS)$-equivalences to $K(I)$-equivalences. This follows from the sequence 
of quasi-isomorphisms
\[
\theta^*(M\Lotimes_S K(IS))\simeq\theta^*(M \Lotimes_S S \Lotimes_R K(I))\simeq\theta^*(M) \Lotimes_R K(I).
\]
Secondly, we need to check that the derived unit on $R$, 
namely the map $\theta\colon R \to S$ is a $K(I)$-equivalence. 
By Lemma~\ref{lem:derived completion imples K equivalence} this is equivalent to 
checking that the map $\Lambda_IR \to \Lambda_IS$ is a quasi-isomorphism 
which holds by hypothesis. 
Therefore the Left Localization Principle applies and the result follows.

For the other two vertical adjunctions we discuss the case for $L_0^I$-completion and note that the version for adic completion is analogous. Note that the adjunction $$\begin{tikzcd}\LR \arrow[r, yshift=1mm, "L_0^J\exta"] & \LS \arrow[l, yshift=-1mm, "\resa"]\end{tikzcd}$$ is well defined since $\resa$ sends $L_0^J$-complete modules to $L_0^I$-complete modules by Lemma~\ref{lem:change of base}. 
Using Proposition~\ref{prop:L properties}(c), one verifies that the left adjoint is the composite $L_0^J\exta$ as claimed. This adjunction is a symmetric monoidal Quillen equivalence by the 2-out-of-3 property of Quillen equivalences, since the horizontals are Quillen equivalences by Theorem~\ref{thm:main1}.
\end{proof}

\begin{rem}
The hypothesis that $J$ is weakly pro-regular in the previous theorem, is equivalent to the hypothesis that $IS$ is weakly pro-regular by~\cite[3.3]{Schenzel03} since $\sqrt{IS} = \sqrt{J}$. In particular, if $\theta\colon R \to S$ is flat and $I$ is weakly pro-regular, then it follows that $J$ is weakly pro-regular by~\cite[3.6]{Yekutieli}. 
\end{rem}

\begin{rem}\label{rem:MGM}
The categories of derived torsion and derived complete complexes are equivalent by the MGM equivalence, see~\cite[2.1]{DwyerGreenlees02} and~\cite[1.1]{MGM} for more details. As such, the previous theorem also shows that the categories of $I$-torsion $R$-modules and $J$-torsion $S$-modules are Quillen equivalent if and only if any of the equivalent conditions in Proposition~\ref{prop:equivconditions} hold.
\end{rem}

We note a straightforward corollary of the previous theorem.
\begin{cor}
Let $\theta\colon R \to S$ be a map of commutative rings and let $I$ and $J$ be weakly pro-regular ideals of $R$ and $S$ respectively such that $\sqrt{IS} = \sqrt{J}$. If $S$ is derived complete then the vertical adjunctions in the diagram shown in Figure~\ref{fig:diagram} are Quillen equivalences if and only if $S$ is isomorphic to $R_I^\wedge$ as a ring.
\end{cor}

\subsection{Change of base along the completion}\label{sec:completionchangeofbase}
Finally we discuss the result of Theorem~\ref{thm:main2} for the completion ring map $\gamma\colon R \to R_I^\wedge$ where $R$ is a commutative ring with an ideal $I$. We denote the ideal $I\cdot R_I^\wedge$ by $\widehat{I}$. 
We now briefly discuss when the ideals $I$ and $\widehat{I}$ are weakly pro-regular.
\begin{lem}\label{lem:wpr along map}
Let $R$ be a commutative ring with an ideal $I$. The ideals $I$ of $R$ and $\widehat{I}$ of $R_I^\wedge$ are weakly pro-regular if either
\begin{itemize}
\item[(a)] $R$ is Noetherian, or,
\item[(b)] the ring map $\gamma\colon R \to R_I^\wedge$ is flat and $I$ is weakly pro-regular.
\end{itemize}
\end{lem}
\begin{proof}
If $R$ is Noetherian, then so is $R_I^\wedge$. Since any ideal in a Noetherian ring is weakly pro-regular, it follows that both $I$ and $\widehat{I}$ are weakly pro-regular. Part (b) follows from~\cite[3.6]{Yekutieli}. 
\end{proof}

\begin{rem}
There are other examples which are not encompassed by the previous lemma, where both $I$ and $\widehat{I}$ are weakly pro-regular, see~\cite[7.2]{Yekutieli18}. 
\end{rem}

Theorem~\ref{thm:main2} shows that provided $\widehat{I}$ is weakly pro-regular, each of the categories of $I$-complete $R$-modules is symmetric monoidally Quillen equivalent to the categories of $\widehat{I}$-complete $R_I^\wedge$-modules. In particular, this recovers a result of Sather-Wagstaff and Wicklein~\cite[4.13]{SWW16}. We also note another generalization of the work of Sather-Wagstaff and Wicklein by Shaul who extended the result to weakly pro-regular ideals, see~\cite[2.13]{Shaul17}.

\bibliographystyle{plain}
\bibliography{references}

\begin{thebibliography}{10}

\bibitem{AJL}
L.~Alonso~Tarr\'{\i}o, A.~Jerem\'{\i}as~L\'{o}pez, and J.~Lipman.
\newblock Local homology and cohomology on schemes.
\newblock {\em Ann. Sci. \'{E}cole Norm. Sup. (4)}, 30(1):1--39, 1997.

\bibitem{AM16}
M.~F. Atiyah and I.~G. Macdonald.
\newblock {\em Introduction to commutative algebra}.
\newblock Addison-Wesley Publishing Co., Reading, Mass.-London-Don Mills, Ont.,
  1969.

\bibitem{AvramovFoxbyHalperin03}
H.H. Avramov, H-B. Foxby, and S.~Halperin.
\newblock Differential graded homological algebra.
\newblock Preprint, 2003.

\bibitem{Barthel}
T.~Barthel, D.~Heard, and G.~Valenzuela.
\newblock Local duality in algebra and topology.
\newblock {\em Adv. Math.}, 335:563--663, 2018.

\bibitem{BHV20}
T.~Barthel, D.~Heard, and G.~Valenzuela.
\newblock Derived completion for comodules.
\newblock {\em Manuscripta Math.}, 161(3-4):409--438, 2020.

\bibitem{DwyerGreenlees02}
W.~G. Dwyer and J.~P.~C. Greenlees.
\newblock Complete modules and torsion modules.
\newblock {\em Amer. J. Math.}, 124(1):199--220, 2002.

\bibitem{EKMM}
A.~D. Elmendorf, I.~Kriz, M.~A. Mandell, and J.~P. May.
\newblock {\em Rings, modules, and algebras in stable homotopy theory},
  volume~47 of {\em Mathematical Surveys and Monographs}.
\newblock American Mathematical Society, Providence, RI, 1997.
\newblock With an appendix by M. Cole.

\bibitem{Greenlees01b}
J.~P.~C. Greenlees.
\newblock Tate cohomology in axiomatic stable homotopy theory.
\newblock In {\em Cohomological methods in homotopy theory ({B}ellaterra,
  1998)}, volume 196 of {\em Progr. Math.}, pages 149--176. Birkh\"{a}user,
  Basel, 2001.

\bibitem{GreenleesMay92}
J.~P.~C. Greenlees and J.~P. May.
\newblock Derived functors of {$I$}-adic completion and local homology.
\newblock {\em J. Algebra}, 149(2):438--453, 1992.

\bibitem{GreenleesMay95b}
J.~P.~C. Greenlees and J.~P. May.
\newblock Completions in algebra and topology.
\newblock In {\em Handbook of algebraic topology}, pages 255--276.
  North-Holland, Amsterdam, 1995.

\bibitem{Hovey99}
M.~Hovey.
\newblock {\em Model categories}, volume~63 of {\em Mathematical Surveys and
  Monographs}.
\newblock American Mathematical Society, Providence, RI, 1999.

\bibitem{HoveyStrickland99}
M.~Hovey and N.~P. Strickland.
\newblock Morava {$K$}-theories and localisation.
\newblock {\em Mem. Amer. Math. Soc.}, 139(666):viii+100, 1999.

\bibitem{SAG}
J.~Lurie.
\newblock Spectral algebraic geometry.
\newblock Available from the author's webpage at
  \url{https://www.math.ias.edu/~lurie/papers/SAG-rootfile.pdf}. Last updated
  February 2018.

\bibitem{Matlis78}
E.~Matlis.
\newblock The higher properties of {$R$}-sequences.
\newblock {\em J. Algebra}, 50(1):77--112, 1978.

\bibitem{PolWilliamson}
L.~Pol and J.~Williamson.
\newblock The {L}eft {L}ocalization {P}rinciple, completions, and cofree
  {$G$}-spectra.
\newblock {\em J. Pure Appl. Algebra}, 224(11):106408, 33, 2020.

\bibitem{MGM}
M.~Porta, L.~Shaul, and A.~Yekutieli.
\newblock On the homology of completion and torsion.
\newblock {\em Algebr. Represent. Theory}, 17(1):31--67, 2014.

\bibitem{Positselski16}
L.~Positselski.
\newblock Dedualizing complexes and {MGM} duality.
\newblock {\em J. Pure Appl. Algebra}, 220(12):3866--3909, 2016.

\bibitem{Positselski17}
L.~Positselski.
\newblock Contraadjusted modules, contramodules, and reduced cotorsion modules.
\newblock {\em Mosc. Math. J.}, 17(3):385--455, 2017.

\bibitem{Positselski}
L.~Positselski.
\newblock Remarks on derived complete modules and complexes.
\newblock arXiv:2002.12331, 2020.

\bibitem{Rezk}
C.~Rezk.
\newblock Analytic completion.
\newblock Available from the author's webpage at
  \url{https://faculty.math.illinois.edu/~rezk/analytic-paper.pdf}.

\bibitem{Salch10}
A.~Salch.
\newblock Approximation of subcategories by abelian subcategories.
\newblock arXiv:1006.0048, 2016.

\bibitem{SWW16}
S.~Sather-Wagstaff and R.~Wicklein.
\newblock Extended local cohomology and local homology.
\newblock {\em Algebr. Represent. Theory}, 19(5):1217--1238, 2016.

\bibitem{Schenzel03}
P.~Schenzel.
\newblock Proregular sequences, local cohomology, and completion.
\newblock {\em Math. Scand.}, 92(2):161--180, 2003.

\bibitem{Shaul17}
L.~{Shaul}.
\newblock {Adic reduction to the diagonal and a relation between cofiniteness
  and derived completion}.
\newblock {\em {Proc. Am. Math. Soc.}}, 145(12):5131--5143, 2017.

\bibitem{Simon09}
A.~M. Simon.
\newblock Approximations of complete modules by complete big {C}ohen-{M}acaulay
  modules over a {C}ohen-{M}acaulay local ring.
\newblock {\em Algebr. Represent. Theory}, 12(2-5):385--400, 2009.

\bibitem{Spaltenstein}
N.~Spaltenstein.
\newblock Resolutions of unbounded complexes.
\newblock {\em Compositio Math.}, 65(2):121--154, 1988.

\bibitem{stacks-project}
The {Stacks project authors}.
\newblock The stacks project.
\newblock \url{https://stacks.math.columbia.edu}, 2020.

\bibitem{Vanderpool12}
R.~Vanderpool.
\newblock Category of {$p$}-complete abelian groups.
\newblock {\em Comm. Algebra}, 40(8):2949--2961, 2012.

\bibitem{Yekutieli11}
A.~Yekutieli.
\newblock On flatness and completion for infinitely generated modules over
  {N}oetherian rings.
\newblock {\em Comm. Algebra}, 39(11):4221--4245, 2011.

\bibitem{Yekutieli15}
A.~{Yekutieli}.
\newblock {A separated cohomologically complete module is complete}.
\newblock {\em {Comm. Algebra}}, 43(2):616--622, 2015.

\bibitem{Yekutieli18}
A.~Yekutieli.
\newblock Flatness and completion revisited.
\newblock {\em Algebr. Represent. Theory}, 21(4):717--736, 2018.

\bibitem{Yekutieli}
A.~Yekutieli.
\newblock Weak proregularity, derived completion, adic flatness, and prisms.
\newblock {\em J. Algebra}, 583:126--152, 2021.

\end{thebibliography}
\end{document}